\documentclass[12pt]{article}
\usepackage{amsmath,amsthm,amsfonts}
\usepackage[shortlabels]{enumitem}
\setlist[enumerate,1]{label={\upshape(\roman*)}}
\usepackage{hyperref}
\usepackage{color}

\newcommand{\nnexteq}{\notag\displaybreak[0]\\ &=}
\newcommand{\nexteq}{\displaybreak[0]\\ &=}

\newtheorem{thm}{Theorem}
\newtheorem{lem}[thm]{Lemma}
\newtheorem{prop}[thm]{Proposition}
\theoremstyle{definition}

\newcommand{\R}{\mathbb{R}}
\newcommand{\Z}{\mathbb{Z}}

\newcommand{\x}{\mathbf{x}}
\newcommand{\w}{\mathbf{w}}
\newcommand{\y}{\mathbf{y}}
\newcommand{\z}{\mathbf{z}}
\newcommand{\bv}{\mathbf{v}}
\newcommand{\bu}{\mathbf{u}}
\newcommand{\br}{\mathbf{r}}

\title{A $2$-distance set with $277$ points in the Euclidean space of dimension $23$}
\author{Hong-Jun Ge, Jack Koolen\thanks{J.H. Koolen is partially supported by the National Natural Science Foundation of China (No. 12471335), and the 
Anhui Initiative in Quantum Information Technologies (No. AHY150000).}, Akihiro Munemasa}
\date{\today}

\begin{document}
\maketitle

{\bf Abstract.}
We construct a $2$-distance set with $277$ points in the $23$-dimensional
Euclidean space $\R^{23}$ having distances $2$ and $\sqrt{6}$.

\section{Introduction}

An $s$-distance set in a Euclidean space $\R^d$ is a point set such that pairwise distances between these points admit only $s$ different values.
Classification and construction of few distance sets in Euclidean space
goes back to the pioneering work of Blokhuis 
\cite{MR751955,MR791015}. He showed that the size of a $2$-distance set
in the $d$-dimensional Euclidean space $\R^d$ is bounded by $\binom{d+2}{2}$.
Lison\v{e}k \cite{MR1429084} determined the sizes of the largest
$2$-distance sets in dimensions $d\leq8$.
Since the set of midpoints of edges of a regular $d$-simplex is a
$2$-distance set of size $\binom{d+1}{2}$, the main research interest lies
in the existence problem of a $2$-distance set whose size exceeds $\binom{d+1}{2}$.
No construction of such sets has been known for $d>8$.

In this paper, we construct a $2$-distance set in 
$\R^{23}$ of size $277$, one larger than the lower bound
$\binom{24}{2}$.
Nozaki and Shinohara
\cite{MR4129024} considered maximal $2$-distance sets containing a regular simplex,
but the sets they produced have smaller size in dimension $23$.
Glazyrin and Yu \cite{MR3787558} determined that, 
if a $2$-distance set lies in the unit sphere $S^{22}\subseteq\R^{23}$,
then its size is bounded by $276$.
Note that, in addition to the standard example of 
the set of $\binom{24}{2}$ midpoints of edges of a regular $23$-simplex,
there are many other $2$-distance sets in $S^{22}$ with the same size.
The unique regular two-graph on $276$ vertices (see \cite{MR384597},
and see \cite[Section~1.5]{BCN} for more information on two-graphs)
gives a set of equiangular lines
in $\R^{23}$, and any set of representing unit vectors forms a $2$-distance set
of size $276$. 
We show that one of the sets of representing unit vectors 
can be used to construct a
maximal $2$-distance set in $\R^{23}$ of size $277$, which cannot be extended to a larger $2$-distance set in $\R^{23}$.

\section{The Construction}
Let $\R^d$ be the $d$-dimensional linear space over $\R$ equipped with the inner product $\langle \bu,\bv\rangle=
\sum_{i=1}^{d}u_iv_i$. As usual, we define the square norm 
$\|\bu\|^2=\langle \bu,\bu\rangle$.

In \cite[Theorem 3.4]{MR384597}, Goethals and Seidel provided a graph $\overline{\Gamma}$ in the switching class of the unique regular two-graph on $276$ vertices, whose Seidel matrix $S=J-I-2A(\overline{\Gamma})$ has spectrum $$\mathrm{Spec}(S)=\{[55]^{23},[-5]^{253}\},$$
where $J$ is the all-ones matrix and $A(\overline{\Gamma})$ is the
adjacency matrix of $\overline{\Gamma}$.
The complement $\Gamma$ of $\overline{\Gamma}$ can be constructed directly
as follows.
Let 
\[X=\{(a,i)\mid a\in \mathbb{F}_3,\;1\leq i\leq 11\},\]
and we consider $X$ naturally as the complete multipartite graph
with $11$ parts, each of which has $3$ vertices.
Let $C$ be the ternary Golay code, and let
$Y=C^\perp$, the dual code of $C$. 
A description of $C$ and $C^\perp$ can be found in
\cite[Lemma 4.3 and Remark 4.4]{MR384597}.
Since $\dim C=6$, we have $|Y|=243$.
We construct the graph $\Gamma$ with vertex set $X\cup Y$ by joining
$(a,i)\in X$ and $y\in Y$ by an edge if $y_i\neq a$, and also joining
$y\in Y$ and $y'\in Y$ by an edge if $y-y'$ has weight $6$, that is,
$y-y'$ has $6$ out of $11$ nonzero coordinates.
Now $S=2A(\Gamma)+I-J$ is the desired Seidel matrix.

From the construction, together with the fact that $C^\perp$
has $132$ vectors of weight $6$, we see that
the partition $X\cup Y$ of the vertex set of $\Gamma$ is equitable, 
with quotient matrix
$$
	\begin{bmatrix}
		30&162\\
		22&132
	\end{bmatrix}.
$$
Thus, $\theta_1=81+\sqrt{6165}$ and $\theta_2=81-\sqrt{6165}$
are two eigenvalues of $A(\Gamma)$.

For each $\alpha\in\{55,-5\}$,
let $V_\alpha$ be the eigenspace of $S$ corresponding to the eigenvalue $\alpha$, and set $W_\alpha:=V_\alpha\,\cap\,\mathbf{1}^{\bot}$, where $\mathbf{1}$ denotes the all-one vector.
Since $\dim W_\alpha\ge \dim V_\alpha-1$, it follows that $\Gamma$ has spectrum $\{[27]^{22},[-3]^{252},[\theta_1]^1,[\theta_2]^1\}$.
Therefore, $A(\Gamma)+3I$ is positive semi-definite with rank $24$, and there exists a set of vectors $\mathbf{V}=\{\mathbf{v}_u\mid u\in X\,\cup\, Y\}$ in $\R^{24}$ satisfying
$\langle \mathbf{v}_u,\mathbf{v}_v\rangle=(A(\Gamma)+3I)_{uv}$
for $u, v\in X\cup Y$. It follows that $\mathbf{V}$
forms a $2$-distance set in $\R^{24}$ consisting of $276$ points
\begin{align}
\|\bv\|^2&=3\quad\quad\ \ \ (\bv\in \mathbf{V}),\notag\\
\|\bv-\bv'\|^2&\in\{4,6\}\quad(\bv,\bv'\in \mathbf{V},\;
\bv\neq\bv').\label{xy}
\end{align}

In what follows, 
we identify the vertices of $\Gamma$ with the corresponding vectors in $\mathbf{V}$.
Then for $\bu,\bv\in \mathbf{V}=X\cup Y$,
\[\langle \bu,\bv\rangle=
\begin{cases}
3&\text{if $\bu=\bv$,}\\
1&\text{if $\bu$ and $\bv$ are adjacent in $\Gamma$,}\\
0&\text{otherwise.}
\end{cases}\]
In \cite[Lemma 4.1]{MR4331063},
Cao~\emph{et~al.} found a vector $\mathbf{r}\in \R^{24}$ such that $\langle \br,\br\rangle=2$ 
and $\langle \br, \bv\rangle=1$ for all $\bv\in \mathbf{V}$. 
The vector $\br$ is called the \emph{switching root}.
More explicitly, $\br$ can be defined as follows.
Let $\{\x_1,\x_2,\x_3\}\subset X$ be the set of vectors corresponding to
one of the $11$ parts of the complete multipartite graph on $X$. Then
these three vectors are pairwise orthogonal. 
Define
\[\br=\x_1+\x_2+\x_3-\frac{4}{33}\sum_{\x\in X}\x+\frac{1}{81}\sum_{\y\in Y}\y.\]
It can be easily verified that $\mathbf{V}$ is contained in the
affine hyperplane 
$\{\bv\in\R^{24}\mid \langle \bv,\br\rangle=1\}\cong\R^{23}$
in $\R^{24}$.

Define
\begin{equation*}
\mathbf{u}=\x_1+\x_2+\x_3-\br.
\end{equation*}
It follows that $\langle \br,\mathbf{u}\rangle=1$ and $\langle\bu,\bu\rangle=5$. 

\begin{thm}
The set $Z:=\{\mathbf{u}\}\cup X\cup Y$ is a $2$-distance set of size
$277$ with squared distances $\{4,6\}$,
contained in the affine hyperplane $\{\bv\in\R^{24}\mid \langle \bv,\br\rangle=1\}\cong\R^{23}$.
\end{thm}
\begin{proof}
In view of \eqref{xy}, the set $X\cup Y$ is a $2$-distance set of size
$276$ with squared distances $\{4,6\}$.
It suffices to prove
$$
\|\bu-\z\|^2\in\{4,6\}\quad(\z\in X\cup Y).$$
Note that $\langle \mathbf{u},\mathbf{u}\rangle=5$ and
\begin{align*}
\langle \bu,\z\rangle&=\langle \x_1+\x_2+\x_3-\br,\z\rangle\notag
\nexteq
\begin{cases}
2&\text{if $\z\in X$,}\\
1&\text{if $\z\in Y$.}
\end{cases}
\end{align*}
It follows that
\begin{align*}
\|\mathbf{u}-\z\|^2&=5+3-2\langle \mathbf{u},\z\rangle
\nexteq
\begin{cases}
4&\text{if $\z\in X$,}\\
6&\text{if $\z\in Y$.}
\end{cases}
\end{align*}
Therefore,
$\{\mathbf{u}\}\cup X\cup Y$ is a $2$-distance set
having squared distances $\{4,6\}$.
\end{proof}

We remark that the vector $\mathbf{u}$ is independent of the choice of the part.
This can be shown in a more general setting.

\begin{lem}\label{lem:1}
	Let $\mathbf{a}_1,\dots,\mathbf{a}_n,\mathbf{b}_1,\dots,\mathbf{b}_n \in \mathbb{R}^d$ 
	be vectors whose Gram matrix is 
	$\begin{bmatrix} nI & J \\ J & nI \end{bmatrix}$, 
	where $J$ denotes the $n \times n$ matrix with all entries equal to $1$. Then
	\[
	\sum_{i=1}^n \mathbf{a}_i = \sum_{i=1}^n \mathbf{b}_i.
	\]
\end{lem}
\begin{proof}
By the Gram matrix,
we have
\begin{align*}
\left\|\sum_{i=1}^n \mathbf{a}_i-\sum_{i=1}^n \mathbf{b}_i\right\|^2&=
\left\|\sum_{i=1}^n \mathbf{a}_i\right\|^2+\left\|\sum_{i=1}^n \mathbf{b}_i\right\|^2
-2\sum_{i=1}^n \sum_{j=1}^n \langle \mathbf{a}_i,\mathbf{b}_j\rangle
\nexteq
n^2+n^2-2n^2
\nexteq
0.
\end{align*}
\end{proof}

Setting $n=3$ in the above lemma, we see that
$\mathbf{u}=\x'_1+\x'_2+\x'_3-\br$ holds whenever
$\{\x'_1,\x'_2,\x'_3\}\subseteq X$
is a set of one of the $11$ parts.

\section{Maximality}

In this section, we show that
our $2$-distance set $Z$ of size $277$ in 
the affine hyperplane 
$H=\{\bv\in\R^{24}\mid \langle \bv,\br\rangle=1\}\cong\R^{23}$ is maximal, 
in the sense that it cannot be extended to a larger $2$-distance set in $\R^{23}$.
However, the $2$-distance set $Z$ is not maximal in $\R^{24}$, since
$Z\cup\{\frac{1+\sqrt{3}}{2} \br,\frac{1-\sqrt{3}}{2} \br\}$ is a larger $2$-distance set. Indeed,
for $\z\in Z$, we have
\begin{align*}
\|\z-\frac{1\pm\sqrt{3}}{2}\br\|^2
&=
\|\z\|^2+2+\sqrt{3}
-(1+\sqrt{3})\langle\z,\br\rangle
\nexteq
\|\z\|^2+1
\nexteq
\begin{cases}
6&\text{if $\z=\bu$,}\\
4&\text{otherwise,}
\end{cases}
\intertext{and}
\|\frac{1+\sqrt{3}}{2} \br-\frac{1-\sqrt{3}}{2} \br\|^2
&=\|\sqrt{3}\br\|^2
\nexteq
6.
\end{align*}
It turns out that this is the only way to add vectors in $\R^{24}$ to $Z$ to maintain
the $2$-distance property, as shown in the following proposition.

\begin{prop}\label{prop:max}
The $2$-distance set $Z$ constructed above is not contained in any larger
$2$-distance set in $H$. The only $2$-distance sets in $\R^{24}$
containing $Z$ properly are
\[
Z\cup\{\frac{1+\sqrt{3}}{2} \br\},\;
Z\cup\{\frac{1-\sqrt{3}}{2} \br\},\text{ and }
Z\cup\{\frac{1+\sqrt{3}}{2} \br,\frac{1-\sqrt{3}}{2} \br\}.
\]
\end{prop}
\begin{proof}
We only prove the second statement, as the first statement follows from the second.
Let $Z'=Z-\bu=\{\z-\bu\mid \z\in Z\}$.
It suffices to show that 
the only $2$-distance set in $\R^{24}$ which properly 
contains $Z'$ is $Z'\cup\{\frac{1+\sqrt{3}}{2} \br-\bu\}$.
Since every pair of vectors in $Z'$ has integer inner product, the $\Z$-span
$M$ of $Z'$ is an integral lattice in $\br^\perp\cong\R^{23}$.

Suppose that there exists a vector $\w\in\R^{24}$ such that $Z'\cup\{\w\}$ is
a $2$-distance set properly containing $Z'$. Since $0\in Z'$, we have
$\langle\w,\w\rangle=4$ or $6$.
Let $\bv$ be the orthogonal projection of $\w$ to $\br^\perp$.
Then $\|\bv\|^2\leq\|\w\|^2\leq6$ and, for $\z'\in Z'$,
\begin{align}
\langle\z',\bv\rangle&=\langle\z',\w\rangle
\nnexteq
\frac12(\|\z'\|^2+\|\w\|^2-\|\z'-\w\|^2)
\notag\\&\in
\begin{cases}
2+\frac12(\|\z'\|^2-\{4,6\})&\text{if $\|\w\|=4$,}\\
3+\frac12(\|\z'\|^2-\{4,6\})&\text{if $\|\w\|=6$}
\end{cases}
\nnexteq
\begin{cases}
\{\frac12\|\z'\|^2,\frac12\|\z'\|^2-1\}
&\text{if $\|\w\|=4$,}\\
\{\frac12\|\z'\|^2,\frac12\|\z'\|^2+1\}
&\text{if $\|\w\|=6$.}
\end{cases}
\label{1234}
\end{align}
In particular, $\langle\z',\bv\rangle\in\Z$ for all $\z'\in Z'$, so
$\bv$ belongs to the dual lattice
\[M^*=\{\x\in \br^\perp\mid \forall \z'\in M,\;\langle\x,\z'\rangle\in\Z\}\]
of $M$. 
By Magma \cite{magma}, it can be verified that $M^*$ contains
$16689170$ nonzero vectors $\bv$ with $\|\bv\|^2\leq6$.
Among them, the only vector $\bv$ that satisfies \eqref{1234} for all $\z'\in Z'$
is $\bv=\frac12\br-\bu$, in which case $\w=\frac{1\pm\sqrt{3}}{2}\br-\bu$.
It follows that the only $2$-distance sets in $\R^{24}$
containing $Z$ properly are the three sets in the statement.
\end{proof}

Lison\v{e}k \cite{MR1429084} showed that there exists a $2$-distance
set in the Euclidean space $\R^7$ consisting of $29=\binom{8}{2}+1$
points. Since $277=\binom{24}{2}+1$, our set may be regarded as an analogue
of Lison\v{e}k's example in $\R^{23}$.
Also, Lison\v{e}k \cite[Theorem~4.4]{MR1429084} showed that
a $2$-distance set of size $29$ in $\R^7$ is unique up to isometry and scaling.
It would be interesting to know if the analogous statement holds for
a $2$-distance set of size $277$ in $\R^{23}$.

Lison\v{e}k \cite{MR1429084} also constructed a $2$-distance set of size 
$45 = \binom{10}{2}$ in $\mathbb{R}^8$, 
which contains the $2$-distance set of size $29$ in $\mathbb{R}^7$ as a subset.
Proposition~\ref{prop:max} shows that
our $277$ points cannot be extended to 
$325=\binom{26}{2}$ points in $\mathbb{R}^{24}$ as a $2$-distance set.

It remains as a challenging problem to decide whether
there exists a $2$-distance set 
in $\R^{24}$ of size $325=\binom{26}{2}$.

\subsection*{Acknowledgements}
The work of the second and third author
was supported by the Research Institute for Mathematical Sciences, an
International Joint Usage/Research Center located in Kyoto University.
We thank Wei-Hsuan Yu for helpful discussions. We also thank the anonymous referees for their valuable comments and suggestions, which helped improve the presentation of this paper.

\subsection*{Data Availability Statement}
No data was used for the research described in the article.

\appendix
\section{Magma code}

\begin{verbatim}
Z:=Integers();
F:=GF(3);
C:=GolayCode(F,false);
X:=[ < a,b > : a in F, b in [1..11] ];
#X eq 33;
EX:={ {i,j} : i,j in {1..33} 
  | i lt j and X[i][2] ne X[j][2] };
K11x3:=MultipartiteGraph([3:i in [1..11]]);
tf:=IsIsomorphic(Graph< 33 | EX >,K11x3);
tf;

Y:=[ y : y in Dual(C) ];
#Y eq 243;

EXY:={ {i,j+33} : i in {1..33}, j in {1..243} 
  | Y[j][X[i][2]] ne X[i][1] };
EY:={ {i+33,j+33} : i,j in {1..243} 
  | Weight(Y[i]-Y[j]) eq 6 };
Gamma:=Graph< 276 | EX join EXY join EY >;
VX:={ Vertices(Gamma) | i : i in {1..33} };
VY:={ Vertices(Gamma) | i : i in {34..276} };
[ [ { #(Neighbours(x) meet VX) : x in VX },
  { #(Neighbours(x) meet VY) : x in VX } ],
  [ { #(Neighbours(y) meet VX) : y in VY },
  { #(Neighbours(y) meet VY) : y in VY } ] ]
  eq [ [ {30},{162} ], [ {22},{132} ] ];

J:=Matrix(Z,276,276,[1:i in [1..276^2]]);
I:=ScalarMatrix(276,1);
A:=AdjacencyMatrix(Gamma);
RX<x>:=PolynomialRing(Z);
RX!CharacteristicPolynomial(A)
  eq (x-27)^22*(x+3)^252*(x^2-162*x+396);
S:=2*A+I-J;
Eigenvalues(S) eq { <55, 23>, <-5, 253> };

gram,XX,r:=LLLGram(A+3*I);
Xi:=XX^(-1);
L:=LatticeWithGram(Submatrix(gram,[1..r],[1..r]));
X276:=[ L![ Xi[i,j] : j in [1..r] ] : i in [1..276] ];
roots:=&join{ {r[1],-r[1]} : r in ShortVectors(L,2,2) };
#roots eq 2;
rt:=[ r : r in roots | (X276[1],r) eq 1 ][1];
&and{ (rt,x) eq 1 : x in X276 };
&and{ (x-y,x-y) in {4,6} : x,y in X276 | x ne y };
u:=&+[ X276[i] : i in [1..3] ]-rt;
(u,u) eq 5 and (rt,u) eq 1;
&and{ (u-x,u-x) in {4,6} : x in X276 };
XL:=[ X276[i] : i in [1..33] ];
YL:=[ X276[i] : i in [34..276] ];
&and{ (XL[i],XL[j]) eq 0 : i,j in [1..3] | i ne j };
rhs:=XL[1]+XL[2]+XL[3]-4/33*&+XL+1/81*&+YL;
Parent(rhs)!rt eq rhs;
Y276:={ x-u : x in X276 };
M1:=sub< L | Y276 >;
Rank(M1) eq 23;
MD:=Dual(M1:Rescale:=false);
Y276D:={ MD | y : y in Y276 };
Minimum(MD) eq 5/2;
Ms:=ShortVectors(MD,5/2,6);
2*#Ms eq 16689170;

adm4:=func< v | forall(z){ z : z in Y276D 
  | (z,v) in { i,i-1 } where i:=(z,z)/2 } >;
not &or{ adm4(w[1]) or adm4(-w[1]) : w in Ms };

adm6:=func< v | forall(z){ z : z in Y276D 
  | (z,v) in { i,i+1 } where i:=(z,z)/2 } >;
&join{ { v : v in { w[1],-w[1] } | adm6(v) } : w in Ms } 
  eq { 1/2*rt-u };

// Total time: 788.600 seconds, Total memory usage: 4584.38MB
\end{verbatim}
\end{document}